\documentclass[a4paper,oneside,11pt]{article}
\usepackage[all]{xy}
\usepackage{a4wide}
\usepackage[T1]{fontenc}
\usepackage[ansinew]{inputenc}
\usepackage{lmodern}
\usepackage{graphicx}
\usepackage{amsmath}
\usepackage{amssymb}
\usepackage{amsthm}
\usepackage{amsfonts}
\usepackage{amscd}
\usepackage{setspace}
\usepackage{mathrsfs}
\newtheorem{theorem}{Theorem}
\theoremstyle{plain}

\newtheorem{lemma}{Lemma}

\newcommand{\ca}{$C^*$-algebra}
\newcommand{\sg}{\sigma}
\newcommand{\er}{\eqref}
\newcommand{\CH}{\mathcal{H}}
\newcommand{\raw}{\rightarrow}
\newcommand{\phv}{\varphi}
\newcommand{\lm}{\lambda}
\newcommand{\ps}{\psi}
\newcommand{\varep}{\varepsilon}
\newcommand{\til}{\tilde}

\begin{document}
\title{A bounded transform approach to self-adjoint operators: \\ Functional calculus  and affiliated von Neumann algebras
}
\author{ Christian Budde and  Klaas Landsman \\ \mbox{}\hfill \\
Radboud University Nijmegen\\Institute for
Mathematics, Astrophysics and Particle Physics\\
Heyendaalseweg $135$, $6525$ AJ Nijmegen, The Netherlands\\ \\
christianbudde@student.ru.nl, landsman@math.ru.nl }
\maketitle
\begin{abstract}
Spectral theory and functional calculus for unbounded self-adjoint
operators on a Hilbert space are usually treated through von Neumann's Cayley transform.
Based on ideas of Woronowicz, we redevelop this theory from the point of view of multiplier algebras and the so-called
bounded transform (which establishes
 a bijective correspondence between closed operators
and pure contractions). This also leads to a simple account of the affiliation relation between von Neumann algebras and 
self-adjoint operators.
\end{abstract}
\setcounter{tocdepth}{4}
%\tableofcontents
\section{Introductory overview}
The  theory of  unbounded self-adjoint operators on a Hilbert space was  initiated by von Neumann, partly motivated by mathematical problems of quantum mechanics \cite{vN}. The monograph by Schm\"{u}dgen \cite{Schm} presents an excellent survey of the present state of the art. 

Von Neumann's approach was based on the Cayley transform and in its subsequent development the notion of a spectral measure played an important role, especially in defining a functional calculus. We consider this route a bit indirect and will avoid both by  firstly invoking the \emph{bounded transform} instead of the Cayley transform, i.e., the formal expressions
\begin{eqnarray}
S&=&T\sqrt{I+T^2}^{-1};\\
T&=& S\sqrt{I-S^2}^{-1},
\end{eqnarray}
make rigorous sense and provide a bijective correspondence between  self-adjoint operators $T$ and  self-adjoint  pure contractions $S$ (i.e., $\left\|Sx\right\|<\left\|x\right\|$ for each $x\in\mathcal{H}\setminus\left\{0\right\}$); cf.\ \cite{Kauf,Kol,Schm}. 

\noindent Note that the bounded transform $T\mapsto S$ is an operatorial version of the homeomorphism $\mathbb{R}\cong \left(-1,1\right)$ given by the function $b: \mathbb{R}\rightarrow (-1,1)$ and its inverse $u: (-1,1)\rightarrow\mathbb{R}$, defined by
\begin{eqnarray}
b(x)&=&\frac{x}{\sqrt{1+x^2}};\label{defb}\\
u(x)&=&\frac{x}{\sqrt{1-x^2}}.\label{defu}
\end{eqnarray}

Secondly, we replace  spectral measures by  simple arguments using multiplier algebras. 
Our approach is based on the work of Woronowicz \cite{Wor1,Wor2}, whose functional calculus we adopt  and  to some extent complete, at least in the usual context of operators on a Hilbert space (Woronowicz's work was mainly intended to deal with problems involving multiplier algebras and, even more generally, with operators on Hilbert $C^*$-modules \cite{Lanc}). 

If $T$ is bounded (and, by standing assumption, self-adjoint), it is easy to prove the equality
\begin{equation}
C^*(T)=C^*(S),\label{CTCS}
\end{equation}
where $C^*(S)$ is the \ca\ generated within $B(\mathcal{H})$ by $S$ and the unit, etc. Furthermore,
 the spectral mapping theorem implies that the spectra of $S$ and $T$ are related by
\begin{eqnarray}
\sigma(T)&=&\left\{\mu(1-\mu^2)^{-\frac{1}{2}}\mid\mu\in\sigma(S)\right\};\\
\sigma(S)&=&\left\{\lambda(1+\lambda^2)^{-\frac{1}{2}}\mid\lambda\in\sigma(T)\right\},
\end{eqnarray}
preserving point spectra. As to the continuous functional calculus, for $S=S^*\in B(\mathcal{H})$ we have the familiar isomorphism
$C(\sigma(S))\stackrel{\cong}{\rightarrow} C^*(S)$, written $g\mapsto g(S)$, given by the spectral theorem. Assuming $T=T^*\in B(\mathcal{H})$,
the same applies to $T$. These calculi are related by
\begin{equation}
f(T)=(f\circ u)(S), \label{defgT}
\end{equation}
where $f\in C(\sigma(T))$, so that $f\circ u \in C(\sigma(S))$. Self-adjointness is preserved, in that 
\begin{equation}
f(T)^*=f^*(T),\label{gTgT}
\end{equation}
where $f^*(x)=\overline{f(x)}$. In particular, if $f$ is real-valued, then $f(T)$ is self-adjoint.
At the level of von Neumann algebras, defining $W^*(S)=C^*(S)''$ and similarly for $T$, eq.\ \er{CTCS}  gives
\begin{equation}
W^*(T)=W^*(S). \label{WTWS}
\end{equation}
 The functional calculus $f\mapsto f(T)$ may then  be extended to bounded Borel functions $f$ on $\sigma(T)$,  in which case it is still given by  \er{defgT}. We then have $f(T)\in W^*(T)$, whilst \er{gTgT} remains valid; however, instead of the isometric property $\|f(T)\|=\|f\|_{\infty}$ for continuous $f$, we now have $\|f(T)\|\leq\|f\|_{\infty}$ (where $\|\cdot\|_{\infty}$ is the supremum-norm). See, e.g., \cite{Ped}.

Our aim is to generalize these results to the case where $T$ is unbounded.  This indeed turns out to be possible, so that our main results are as follows. Throughout the remainder of this paper we assume that $T^*=T$ is possibly unbounded, with bounded transform $S$. 
\begin{theorem}\label{T1}
The (point) spectra of $T$ and its bounded transform $S$ are related by
\begin{eqnarray}
\sigma(T)&=&\left\{\mu(1-\mu^2)^{-\frac{1}{2}}:\mu\in\tilde{\sigma}(S)\right\};\label{eq1.6}\\
\sigma(S)&=&\left\{\lambda(1+\lambda^2)^{-\frac{1}{2}}:\lambda\in\sigma(T)\right\}^-, \label{eq1.7}
\end{eqnarray}
where $\mbox{}^-$ denotes the closure in $\mathbb{R}$, and we abbreviate
\begin{equation}\tilde{\sigma}(S)=\sigma(S)\cap\left(-1,1\right).\end{equation}
\end{theorem}
Note that $\tilde{\sigma}(S)=\sigma(S)$ iff $T$ is bounded (in which case $\sg(S)$ is a compact subset of
$(-1,1)$, since $\pm 1\in\sg(S)$ iff $T$ is unbounded). 
We define the following operator algebras within $B(\mathcal{H})$:
\begin{equation}
C^*_{\bullet}(S)=\left\{g(S):g\in C_{\bullet}(\tilde{\sigma}(S))\right\},\label{defC0S}
\end{equation}
where $\bullet$ is $b$, $c$, or $0$, so that we have defined 
$C^*_c(S)$, $C^*_0(S)$, and $C^*_b(S)$.
Notice that $C(\sigma(S))$ consists of all $g\in
C_b(\tilde{\sigma}(S))$ for which $\lim_{y\rightarrow\pm1}{g(y)}$ exists, where this
limit is $0$ if and only if $g\in C_0(\tilde{\sigma}(S))$. Hence we have the inclusions (of which the first set implies the second)
\begin{eqnarray}
&& C_c(\tilde{\sigma}(S))\subseteq C_0(\tilde{\sigma}(S))\subseteq C(\sg(S))\subseteq C_b(\tilde{\sigma}(S));\\
&& C^*_c(S)\subseteq C^*_0(S)\subseteq C^*(S) \subseteq C^*_b(S),
\end{eqnarray}
with equalities iff $T$ is bounded. This means that $g(S)$ is defined for 
$g\in C_0(\tilde{\sigma}(S))$, and hence  \emph{a fortiori} also for $g\in C_c(\tilde{\sigma}(S))$. Consequently,
$f(T)$ may be defined by \er{defgT} whenever $f\in C_0(\sg(T))$, including  $f\in C_c(\sg(T))$. To pass to the larger class
$f\in C_b(\sg(T))$, we define $C^*_0(S)\mathcal{H}$ as the linear span of
all vectors of the form $g(S)\psi$, where $g\in C_0(\tilde{\sigma}(S))$ and $\psi\in\mathcal{H}$.
Then $C^*_0(S)\mathcal{H}$ is dense in $\mathcal{H}$ (Lemma \ref{Buddelemma}).
 In the spirit of Woronowicz \cite{Lanc,Wor1}, we then initially  define $f(T)$ for $f\in C_b(\sg(T))$ 
 on the domain $C^*_0(S)\mathcal{H}$
 by linear extension of the formula
 \begin{equation}
f_0(T)h(T)\psi=(fh)(T)\psi, \label{deff0T}
\end{equation}
where $h\in C_0(\sg(T))$ and hence also $fh\in  C_0(\sg(T))$, since $C_b(\sg(T))$ is the mutiplier algebra of $C_0(\sg(T))$.
Then $f_0(T)$ is bounded (Lemma \ref{fTbounded}), and we define $f(T)$ as its closure, i.e., 
\begin{equation}
f(T)=f_0(T)^-.\label{fTf0T}
\end{equation}
This also works for $f\in C(\sg(T))$, in which case $f_0(T)$ may no longer be bounded, but remains closable (Lemma \ref{f0Tclosable}), so that we may once again define $f(T)$ as its closure, cf.\ \er{fTf0T}.
We have:
\begin{theorem}\label{T2}
If $f\in C(\sigma(T))$ is real-valued, then $f(T)$ is self-adjoint, i.e., $f_0(T)^-=f_0(T)^*$; more generally,  $f(T)^*=f^*(T)$. Furthermore, the  continuous functional calculus $f\mapsto f(T)$ 
restricts to an isometric $\mbox{}^*$-homomorphism from  $C_0(\sg(T))$ (with supremum-norm) to $C^*(S)$.
\end{theorem}
See also Theorem \ref{T3} . 
In addition, the map $f\mapsto f(T)$
has the reassuring special cases
  \begin{eqnarray}
 \mathbf{1}_{\sg(T)}(T)&=& I;\label{unitsc}\\
\mathrm{id}(T)&=&T; \label{idaisa}\\
(\mathrm{id}-z)^{-1}(T)&=& (T-z)^{-1}, \:\: z\in\rho(T),\label{idsgaminz}
\end{eqnarray}
where $ \mathbf{1}_{\sg(T)}(x)=1$ and $\mathrm{id}(x)=x$  ($x\in\sg(T)$), and therefore does what it is supposed to to.

Finding the right analogue of \er{WTWS} for unbounded $T=T^*$ first requires a redefinition of $W^*(T)$, which is standard \cite{Ped}.
If $T$ is unbounded and  $R\in B(\mathcal{H})$, then we say that $R$ and $T$ commute, written $TR\subset RT$,
 if $R\psi\in \mathcal{D}(T)$ 
and $RT\psi=TR\psi$ for any $\psi\in\mathcal{D}(T)$. Let
$\left\{T\right\}'$ be the set of all bounded operators that commute with $T$. If
$T^*=T$, then 
$\left\{T\right\}'$ is a unital, strongly closed
$*$-subalgebra of $B(\mathcal{H})$, and hence a von Neumann algebra \cite{Ped}.
Its commutant
\begin{equation}
W^*(T)=\left\{T\right\}'', \label{defWstarT}
\end{equation}
is a von Neumann algebra, too. If $T$ is bounded, then $W^*(T)$
is the von Neumann algebra generated by $T$, which coincides with $C^*(T)''$. 
As usual, we  call  a closed unbounded operator $X$  \emph{affiliated} to a von Neumann algebra $A\subset B(H)$, written $X\eta A$,
iff $XR\subset RX$ for each $R\in A'$. For example, if $T^*=T$, then  $T\eta W^*(T)$, and if $T\eta A$, then $W^*(T)\subseteq A$; in other words, $W^*(T)$ is the smallest  von Neumann algebra such that $T$ is affiliated to it. 

As a result of independent interest as well as a lemma for Theorem \ref{T3}, we may then adapt \cite[Lemma 5.2.8]{Ped} to the bounded transform:
\begin{theorem}\label{T4}
Let $A\subset B(H)$ be a von Neumann algebra. Then
$T\eta A$ iff $S\in A$.
\end{theorem}
Denoting the (Banach) space of (bounded) Borel functions on $\sg(T)$ (equipped with the supremum-norm) by $\mathcal{B}_{(b)}(\sigma(T))$, 
we may still define $f(T)$  by \er{defgT} and the usual Borel functional calculus for the bounded transform $S$.
\begin{theorem}\label{T3}
 The map $f\mapsto f(T)$   is a norm-decreasing  $\mbox{}^*$-homomorphism from 
$\mathcal{B}_b(\sigma(T))$ to
\begin{equation}
W^*(T)=W^*(S). \label{WstaraWstarb}
\end{equation}
More generally, if $f\in \mathcal{B}(\sigma(T))$, then $f(T)$ is affiliated with $W^*(T)$. 
\end{theorem}

The remainder of this paper simply consists of the proofs of these theorems.

\medskip
The authors are indebted to Eli Hawkins, Nigel Higson, Jens Kaad, Erik Koelink, Bram Mesland, and Arnoud van Rooij
 for advice (most of which was taken).
\section{Proofs}
This section contains all proofs. We will not repeat the theorems.
\subsection{Proof of Theorem \ref{T1}}
The operator
$\sqrt{1-S^2}$ is a bijection from $\mathcal{H}$ to
$\mathcal{R}(\sqrt{1-S^2})=\mathcal{D}(T)$. Let
$\lambda\in\rho(T)\equiv \mathbb{C}\setminus\sigma(T)$, so that $T-\lambda I$ is a bijection from
$\mathcal{D}(T)$ to $\mathcal{H}$. Thus by composition we have a bijection
$\mathcal{H}\rightarrow\mathcal{H}$; equivalently, $(T-\lambda I)(\sqrt{I-S^2})$ is
invertible, which in turn is equivalent to invertibility of
$S-\lambda\sqrt{I-S^2}$. Thus  $\lambda\in\rho(T)\Longleftrightarrow S-\lambda\sqrt{I-S^2}$ is a bijection, or, expressed 
contrapositively,  $\lambda\in\sigma(T)\Longleftrightarrow
S-\lambda\sqrt{I-S^2}$ is not invertible in $B(\CH)$. This is the case iff
$S-\lambda\sqrt{I-S^2}$ is not invertible in $C^*(S)$, which, using the Gelfand isomorphism
$C^*(S)\cong C(\sg(S))$, in turn is true iff the function
$k_{\lambda}(x)=x-\lambda\sqrt{1-x^2}$  is not invertible in  $C(\sigma(S))$, i.e., iff $0\in\sg(k_{\lambda})$.
Since in $C(X)$ we have 
$\sigma(f)=\mathcal{R}(f)$ (with $X$ a compact Hausdorff space), and 
$\sigma(S)$ is indeed compact and Hausdorff because  
$S$ is bounded, we obtain $\lambda\in\sg(T)$ iff $0\in \mathcal{R}(k_{\lambda})$.
If $\pm 1$ lie in $\sg(S)$ they cannot give rise to this possibility, since  $k_{\lambda}(\pm1)=\pm1$ for each
$\lambda$. Hence we have $0\in \mathcal{R}(k_{\lambda})$ iff $\lm=\mu(1-\mu^2)^{-\frac{1}{2}}$ for some
$\mu\in\sg(S)\cap(-1,1)$, which yields
 \er{eq1.6}. 
 
 The same argument shows that
$\mu\in\sigma(S)\cap\left(-1,1\right)$ comes from $\lambda\in\sigma(T)$. But since
$\sigma(S)$ is compact and hence closed in $\left[-1,1\right]$ we obtain \er{eq1.7}.
\hfill $\Box$
\subsection{Proof of Theorem \ref{T2}}
This proof relies on three lemma's.
\begin{lemma}\label{Buddelemma}
Let $C^*_c(S)\mathcal{H}$ be the linear span of
all vectors of the form $g(S)\psi$, where $g\in C_c(\tilde{\sigma}(S))$ and $\psi\in\mathcal{H}$.
Then $C^*_c(S)\mathcal{H}$ is dense in $H$.
\end{lemma}
\begin{proof}Define $g_n:(-1,1)\raw[0,1]$ by putting  $g_n(x)=0$
 for 
$x\in\left(-1,\frac{1}{n}-1\right]\cup\left[1-\frac{1}{n},1\right)$, $g_n(x)=1$ if
$x\in \left[\frac{2}{n}-1,1-\frac{2}{n}\right]$, and linear interpolation in between.
The ensuing sequence converges
pointwise  to the unit $\mathbf{1}$ on $\left(-1,1\right)$. Restricting
each  $g_n$ to $\tilde{\sigma}(S)$, the continuous
functional calculus gives $g_n(S)\rightarrow\mathbf{1}_{\tilde{\sigma}(S)}$ strongly. Therefore, for any $\ps\in\CH$
we have a sequence  $\ps_n=g_n(S)\psi$ in $C^*_c(S)\mathcal{H}$ such that $\ps_n\raw\ps$. 
\end{proof}

\begin{lemma}\label{fTbounded}
For $f\in C_b(\sg(T))$, define an operator $f_0(T)$ on the domain $C^*_0(S)\mathcal{H}$ by \er{deff0T}. Then $f_0(T)$ is bounded, with bound
\begin{equation}
\left\|f(T)\right\|\leq\left\|f\right\|_{\infty}.\label{faboundunbound}
\end{equation}
\end{lemma}
\begin{proof} Let $\varepsilon>0$. If $h\in C_0(\sg(T))$, then $fh\in C_0(\sg(T))$, so that 
 we can
find a compact subset $K\subset\sg(T)$ such that
$\left|h(x)f(x)\right|<\varepsilon$ for each $x\notin K$. 
Let
$\tilde{h}=h\circ u$, cf.\ \er{defu}; then $\tilde{h}\in  C_0(\til{\sg}(S))$ whenever $h\in C_0(\sg(T))$; in fact,
 we have  an isometric isomorphism 
\begin{equation}
C_0(\sg(T))\stackrel{\cong}{\rightarrow} C_0(\til{\sg}(S)), \:\: h\mapsto h\circ u. \label{ffu}
\end{equation}
 Contractivity of the Borel functional calculus for bounded
operators on $\mathcal{H}$ gives
$$
\|(\widetilde{\mathbf{1}_{K^c}fh})(S)\psi\|
\leq \|(\widetilde{\mathbf{1}_{K^c}fh})(S)\|\|\psi\|\leq\|\widetilde{\mathbf{1}_{K^c}fh}\|_{\infty}\|\psi\|
<\varep \|\psi\|.
$$
Using also the homomorphism property of the Borel functional calculus, we then find
\begin{eqnarray*}
\|(fh)(T)\psi\| &=& \|(\widetilde{fh})(S)\psi\| \\ &=& \|\widetilde{(\mathbf{1}_Kfh})(S)+(\widetilde{fh}-\widetilde{\mathbf{1}_Kfh})(S)\psi\| 
\\ &\leq&  \|(\widetilde{\mathbf{1}_Kfh})(S)\psi\|+\|(\widetilde{\mathbf{1}_{K^c}fh})(S)\ps\|\\
&=&  \|\widetilde{(\mathbf{1}_Kf)}(S)\tilde{h}(S)\psi\|+\|(\widetilde{\mathbf{1}_{K^c}fh})(S)\ps\| \\
&<&  \|\widetilde{(\mathbf{1}_Kf)}\|_{\infty} \|h(T)\ps\| +\varep \|\psi\|,\\\
&\leq &  \|f\|_{\infty}\,  \|h(T)\ps\| +\varep \|\psi\|, 
\end{eqnarray*}
since $\|\widetilde{(\mathbf{1}_Kf)}\|_{\infty}\leq \|\til{f}\|_{\infty}=\|f\|_{\infty}$. Since the last expression above is independent of $K$, we may let $\varep\raw 0$, obtaining boundedness of $f(T)$ as well as \er{faboundunbound}.
\end{proof}
The last claim in Theorem \ref{T2} now  follows from the continuous functional calculus for $S$
and the isometric isomorphism \er{ffu}.
Although isometry may be lost if we go from $C_0(\sg(T))$ to $C_b(\sg(T))$, it easily follows from \er{deff0T} - \er{fTf0T} that 
the map $f\mapsto f(T)$ at least defines a $\mbox{}^*$-homomorphism $C_b(\sg(T))\raw B(H)$. This property will be used after Lemma \ref{Nel} below. 
\begin{lemma}\label{f0Tclosable}
For $f\in C(\sg(T))$, define an operator $f_0(T)$ on the domain $C^*_c(S)\mathcal{H}$ by \er{deff0T}. 
Then $f_0(T)$ is closable. Moreover, if $f$ is real-valued ($f^*=f$),
then $f_0(T)$ is symmetric.
\end{lemma}
\begin{proof}
Suppose that
$h_1(T)\psi_1$ and $h_2(T)\psi_2$ lie in $\mathcal{D}(f_0(T))$. Then we may compute:
\begin{eqnarray}
\left\langle h_2(T)\psi_2,f_0(T)h_1(T)\psi_1\right\rangle&=&\left\langle
\psi_2,\overline{h_2}(T)(fh_1)(T)\psi_1\right\rangle=\left\langle
\psi_2,(\overline{h_2}fh_1)(T)\psi_1\right\rangle;\label{ob1}\\
\left\langle (h_2\overline{f})(T)\psi_2,h_1(T)\psi_1 \right\rangle&=&\left\langle
\psi_2,\overline{(h_2\overline{f})}h_1(T)\psi_1\right\rangle
=\left\langle \psi_2,(\overline{h_2}fh_1)(T)\psi_1\right\rangle.\label{ob2}
\end{eqnarray}
This implies that $\mathcal{D}(f_0(T))\subseteq\mathcal{D}(f_0(T)^*)$
Since  $\mathcal{D}(f_0(T))$ is dense, so is,
$\mathcal{D}(f_0(T)^*)$, which implies that $f_0(T)$ is closable. The second claim is  obvious from \er{ob1} - \er{ob2}.
\end{proof}
\begin{proof} 
To prove Theorem \ref{T2} we use a well-known result of Nelson \cite{Nelson}; see also  
\cite{Ree} (this step was suggested to us by Nigel Higson). For convenience we recall this result (without proof):
\begin{lemma}\label{Nel}
Let $\left\{U(t)\right\}_{t\in\mathbb{R}}$ be a strongly continuous
unitary group of operators on a Hilbert space $\mathcal{H}$. Let
$R:\mathcal{D}(R)\rightarrow\mathcal{H}$ be densely defined and symmetric. Assume
that $\mathcal{D}(R)$ is invariant under $\left\{U(t)\right\}_{t\in\mathbb{R}}$,
i.e. $U(t):\mathcal{D}(R)\rightarrow\mathcal{D}(R)$ for eacht $t$, and also that
$\left\{U(t)\right\}_{t\in\mathbb{R}}$ is strongly differentiable on
$\mathcal{D}(R)$. Then $-idU(t)/dt$ is essentially self-adjoint on
$\mathcal{D}(R)$ and its closure is the self-adjoint generator of
$\left\{U(t)\right\}_{t\in\mathbb{R}}$ (given by Stone's Theorem).
 In particular,  if
$(dU(t)/dt) \psi=iRU(t)\psi$ for each $\psi\in\mathcal{D}(R)$, then
$R$ is essentially self-adjoint. 
\end{lemma}
Set $R=f_0(T)$ for $f\in C(\sigma(T))$, so that
\begin{equation}
\mathcal{D}(R)=C_c^*(S)\CH,\label{domR}
\end{equation}
and for each $t\in\mathbb{R}$ define
$U(t)$ via the (bounded) function $x\mapsto \exp(itf(x))$ on $\sg(T)$, that is, for
 $h\in C_c(\sigma(T))$ and $\psi\in\CH$, we initially define
 \begin{equation}
U_0(t)h(T)\psi=(e^{itf}h)(T)\psi.  \label{defU0}
\end{equation}
Then $U_0$ bounded by Lemma \ref{fTbounded}, and we define $U(t)$ as the closure of $U_0(t)$. 
The remark before Lemma \ref{f0Tclosable} then implies that  $t\mapsto U(t)$ defines a unitary representation of $\mathbb{R}$ on $\CH$. Strong continuity of this representation follows from an $\varep/3$ argument. First, for
\begin{equation}
\phv=h(T)\psi,  \label{phvisTpsi}
\end{equation}
assuming $\|\psi\|=1$ for simplicity,  eqs.\ \er{defU0} and \er{faboundunbound} give 
\begin{equation}
\|U(t)\phv-\phv\|\leq \|e^{itf}h-h\|_{\infty} \leq \| h\|_{\infty} \|e^{itf}-\mathbf{1}\|_{\infty}^{(K)}, \label{Uest1}
\end{equation}
where $K$ is the (compact) support of $h$ in $\sg(T)$. Since the exponential function is uniformly convergent on any compact set,
this gives $\lim_{t\raw 0}\|U(t)\phv-\phv\|=0$ for  $\phv$ of the form \er{phvisTpsi}; taking finite linear combinations thereof gives the same result for any $\phv\in  C_c^*(S)\CH$.
Thus for any $\varep>0$ we can find $\delta>0$ so that $\|U(t)\phv-\phv\|<\varep/3$ whenever $|t|<\delta$.
 For general $\psi'\in H$, we find $\phv\in C_c^*(S)H$ such that $\|\phv-\ps'\|<\varep/3$, and estimate
\begin{eqnarray*}
 \| U(t)\psi'-\ps'\|&\leq & \| U(t)\psi'-U(t)\phv\| + \|U(t)\phv-\phv\|+\|\phv-\ps'\|\\
 &\leq &\varep/3+\varep/3+\varep/3=\varep,
\end{eqnarray*}
since $\| U(t)\psi'-U(t)\phv\|=\|\ps'-\phv\|$ by unitarity of $U(t)$. Thus  $\lim_{t\raw 0}\|U(t)\psi-\psi\|=0$ for 
any $\psi\in\CH$, so that the unitary representation $t\mapsto U(t)$ is strongly continuous. Similarly,
\begin{equation}
\left\|\frac{U(t+s)\phv-U(t)\phv}{s}-iRU(t)\phv\right\|
\leq\left\|\frac{e^{isf}h-h}{s}-ifh\right\|_{\infty},
\end{equation}
assuming \er{phvisTpsi},
so that by the same argument as in \er{Uest1} we obtain 
\begin{equation}
\frac{dU(t)}{dt}\phv=iRU(t)\phv, 
\end{equation}
 initially for any $\phv$ of the form \er{phvisTpsi}, and hence, taking finite sums,  for any
$\phv\in\mathcal{D}(R)$, cf.\ \er{domR}.
The final part of Lemma \ref{Nel} then shows that  $f_0(T)$ is essentially self-adjoint on its domain $C_c^*(S)\CH$.
Its closure $f(T)$ is therefore self-adjoint, and Theorem \ref{T2} is proved. 
\end{proof}

We now prove the examples \er{unitsc} - \er{idsgaminz}, of which the first is trivial. 
 Writing $T_0$ for the operator $\mathrm{id}_0(T)$, 
the definition  \er{deff0T}  gives $$T_0\phv=T\phv$$ for $\phv\in \mathcal{D}(T_0)= C^*_c(S)\mathcal{H}$.
Let $\psi\in \mathcal{D}(T_0^-)$, so that there is a sequence $(\phv_n)$ in $\mathcal{D}(T_0)$ such that $\phv_n\raw\phv$ and
$(T_0\phv_n)$ converges. Since $T$ is closed, it follows that $T_0\phv_n=T\phv_n\raw T\phv$, so that $\phv\in \mathcal{D}(T)$.
Hence $T_0^-\subset T$. Since both operators are self-adjoint, this implies $T_0^-=T$, which proves \er{idaisa}.

The proof of \er{idsgaminz} is easier since $(T-z)^{-1}$ is bounded: 
writing $$f(x)=(x-z)^{-1},$$ where $z\notin\sigma(T)$ is fixed and $x\in\sigma(T)$, we have 
$$f_0(T) h(T)\psi=(fh)(T)\psi=(T-z)^{-1}h(T)\psi,$$ and hence $$f_0(T)\phv=(T-z)^{-1}\phv$$ for any $\phv\in  \mathcal{D}(f_0(T))= C^*_c(S)\mathcal{H}$.
So if  $\phv_n\raw\phv$ for $\phv\in \CH$ and $\phv_n\in  \mathcal{D}(f_0(T))$, boundedness and hence continuity of the resolvent implies
$$f(T)\phv=\lim_{n\rightarrow\infty} f_0(T)\phv_n=\lim_{n\rightarrow\infty} (T-z)^{-1}\phv_n=(T-z)^{-1}\phv.$$
\subsection{Proof of Theorem \ref{T4}}
The first step consists in the observation that $T\eta A$ iff
 $TU\subset UT$ (or, equivalently, $UTU^*=T$) merely for each unitary $U\in A'$, which is well known \cite{Zsi}.
  
 The second step is to show that  $TU\subset UT$  iff $SU=US$ for any unitary $U$. This is a simple computation.
 First suppose that $UTU^*=T$. Then:
\begin{align*}
U(1+T^2)^{-1}U^*&=(U(1+T^2)U^*)^{-1}=((U+UT^2)U^*)^{-1}\\
&=(UU^*+UT^2U^*)^{-1}=(1+UTU^*UTU^*)^{-1}\\ &=(1+T^2)^{-1}.
\end{align*}
If $R$ is bounded and positive, then $UR=RU$ iff $U\in C^*(R)'$, and since
$\sqrt{R}\in C^*(R)$ by the continuous functional calculus, we also have
 $U\sqrt{R}=\sqrt{R}U$. Consequently, 
$$USU^*=U\left(T\sqrt{(1+T^2)^{-1}}\right)U^*=\left(UTU^*\right)\left(U\sqrt{(1+T^2)^{-1}}U^*\right)=T\sqrt{(1+T^2)^{-1}}=S.$$

Similarly, if $SU=US$, then
$$
UTU^*=US\sqrt{1-S^2}^{-1}U^*
=SU\sqrt{1-S^2}^{-1}U^*=S\left(U\sqrt{1-S^2}U^*\right)^{-1}=S\sqrt{1-S^2}^{-1}=T.
$$
Thirdly, as in the first step,  $SU=US$ for any unitary $U\in A'$ iff $S\in A''=A$.
\hfill $\Box$
\subsection{Proof of Theorem \ref{T3}}
Eq.\ \er{WstaraWstarb} in Theorem \ref{T3} follows from Theorem \ref{T4}: taking $A=W^*(T)$, so that $T\eta A$,  yields $S\in W^*(T)$, and hence $W^*(S)\subseteq W^*(T)$. On the other hand, taking $A=W^*(S)$, in which case $S\in A$, gives $T\eta W^*(S)$, and hence $W^*(T)\subseteq W^*(S)$. 

Similar to \er{ffu}, we have an isometric isomorphism 
\begin{equation}
\mathcal{B}_b(\sg(T))\stackrel{\cong}{\rightarrow} \mathcal{B}_b(\til{\sg}(S)), \:\: h\mapsto h\circ u, \label{ffu2}
\end{equation}
so that the first claim of Theorem \ref{T3} follows from the Borel functional calculus for the bounded operator $S$ \cite{Ped}. 
The proof of the last one is, \emph{mutatis mutandis}, practically  the same as in \cite[Theorem 5.3.8]{Ped}, so we omit the details; see \cite{Budde}.  \hfill $\Box$

As explained in \cite[\S 5.3]{Ped}, there exists a  Borel measure $\mu$ on $\sg(T)$ such that
the map $f\mapsto f(T)$ may also be seen as a so-called essential $\mbox{}^*$-homomorphism from
$\mathcal{B}(\sg(T))/\mathcal{N}(\sg(T))$ into the $\mbox{}^*$-algebra of normal operators affiliated with $W^*(T)$,
where $\mathcal{N}(\sg(T))$ is the set of $\mu$-null functions on $\sg(T)$. This remains true in our approach, with the same proof \cite{Budde}.
\section{Epilogue}
Let us finally note that although this paper was inspired by the work of Woronowicz, the $C^*$-algebraic affiliation relation he defines in  \cite{Wor1} (as did, independently, also Baaj and Julg \cite{Baaj})  has not been used here. If we call his relation $\eta'$ to avoid confusion with the $W^*$-algebraic relation $\eta$ we do use, if $A\subset B(\mathcal{H})$ we have $T\eta' A\Rightarrow T\in A$ (and hence $T$ is bounded), cf.\ \cite[Prop.\ 1.3]{Wor1}. Woronowicz does not define a \ca ic counterpart of the von Neumann algebra $W^*(T)$, but it might be reasonable to define $C^*(T)$ as the smallest \ca\ $A$ in $B(\mathcal{H})$ such that $T\eta' A$. It follows from \cite[Example 4]{Wor1} that this
 would give 
$C^*(T)=C^*_0(S)$,
 as defined in \eqref{defC0S}. This \ca\ contains $S$ (and hence $T$) if and only if $T$ is bounded, in which case  $C^*_0(S)=C^*(S)$ and hence $C^*(T)=C^*(S)$, as in our approach, cf.\ \er{CTCS}. Also in general (i.e., if $T$ is possibly unbounded), the bicommutant $C^*(T)''$ coincides with $W^*(T)$ as defined in the usual way \eqref{defWstarT} this follows from $C^*_0(S)''=C^*(S)''=W^*(S)$ and \er{WTWS}.

 Of course, we could also redefine $\eta'$, now calling it $\eta''$, by stipulating that  $T\eta'' A$  whenever $S\in A$, and redefine $C^*(T)$ accordingly (i.e., as the smallest \ca\ $A$ in $B(\mathcal{H})$ such that $T\eta'' A$). This would give  \er{CTCS}
  even if $T$ is unbounded, though in a somewhat empty way.

\end{document}